\documentclass{article}
\usepackage{amsfonts,amsmath,amsthm,amssymb}
\usepackage{graphics, epsfig}
\usepackage{color}
\usepackage{appendix}
\usepackage{ulem}
\usepackage[makeroom]{cancel}
\usepackage{fancyhdr}
\usepackage{centernot}
\usepackage{tikz-cd}
\usepackage{mathtools}
\usepackage{ stmaryrd }

 \usepackage[usenames,dvipsnames]{pstricks}
 \usepackage{pst-grad} % For gradients
 \usepackage{pst-plot} % For axes
\allowdisplaybreaks

\let\TeXchi\chi
\newbox\chibox
\setbox0 \hbox{\mathsurround0pt $\TeXchi$}
\setbox\chibox \hbox{\raise\dp0 \box 0 }
\def\chi{\copy\chibox}

\newtheorem{proposition}{Proposition}[section]
\newtheorem{theorem}{Theorem}[section]
\newtheorem{definition}{Definition}[section]
\newtheorem{example}{Example}[section]
\newtheorem{lemma}{Lemma}[section]
\newtheorem{corollary}{Corollary}[section]
\newtheorem{remark}{Remark}[section]

\numberwithin{equation}{section}
\numberwithin{theorem}{section}
\numberwithin{definition}{section}
\numberwithin{example}{section}
\numberwithin{proposition}{section}
\numberwithin{lemma}{section}
\numberwithin{remark}{section}
\DeclareMathOperator{\Ob}{Ob}
\DeclareMathOperator{\Arg}{Arg}
\DeclareMathOperator{\Spec}{Spec}
\DeclareMathOperator{\Spa}{Spa}
\DeclareMathOperator{\Prim}{Prim}
\DeclareMathOperator{\Pg}{Pg}
\DeclareMathOperator{\Str}{Str}
\newcommand\blfootnote[1]{%
  \begingroup
  \renewcommand\thefootnote{}\footnote{#1}%
  \addtocounter{footnote}{-1}%
  \endgroup
}
\begin{document}
%%%%%%%%%%%%%%%%%%%%%%%%%%%%%%%%%%%%%%%%%%%%%%%%%%%
\title{Primigraph spaces}
%%%%%%%%%%%%%%%%%%%%%%%%%%%%%%%%%%%%%%%%%%%%%%%%%%%
\author
{Manuel Norman}
%%%%%%%%%%%%%%%%%%%%%%%%%%%%%%%%%%%%%%%%%%%%%%%%%%%
\date{}
\maketitle
%%%%%%%%%%%%%%%%%%%%%%%%%%%%%%%%%%%%%%%%%%%%%%%%%%%
\begin{abstract}
\noindent In this paper we introduce primigraph spaces, which are topological spaces together with a sheaf of $C^*$-algebras that can be covered by some $\Prim A$'s, that is, by the primitive spectra of some $C^*$-algebras endowed with Jacobson topology and together with the sheaf of bounded continuous
functions on them. This notion is analoguous, in some sense, to the ones of schemes and perfectoid spaces. Our first main result here is that every topological space gives rise to a primigraph space; this will imply, thanks to Theorem \ref{Thm:4.1}, that primigraph spaces also constitute a new kind of topological invariant.
%%%%%%%%%%%%%%%%%%%%%%%%%%%%%%%%%%%%%%%%%%%%%%%%%%%
\end{abstract}
%%%%%%%%%%%%%%%%%%%%%%%%%%%%%%%%%%%%%%%%%%%%%%%%%%%
\blfootnote{Author: \textbf{Manuel Norman}; email: manuel.norman02@gmail.com
}
%%%%%%%%%%%%%%%%%%%%%%%%%%%%%%%%%%%%%%%%%%%%%%%%%%%
\section{Introduction}
In [1] we introduced the notion of structured space, which has been developed in various other papers. In particular, in [2] we have studied some cohomologies for structured spaces arising from their corresponding poset (see Section 4 in [1]). In [3] we have shown that some ideas from [1-2] can be fruitfully applied to schemes and perfectoid spaces, obtaining some stratifications that allow us to define conical calculus on them. In the cited paper we also noticed that similar reasonings can be applied in other contexts: following this direction, here we use again the map $h$ defined in [1] in order to obtain some posets. We will then construct some graphs in a natural way, so that we can consider their graph $C^*$-algebras. Thanks to some results, we will show that these give rise to a primigraph space, which will be denoted by $\Pg(\cdot)$. More generally, a primigraph space is a topological space together with a sheaf of $C^*$-algebras on it which can be covered by primitive spectra of some $C^*$-algebras (endowed with Jacobson topology and together with the sheaf of bounded functions). Our second main result is that these primigraph spaces associated to topological spaces are actually a new kind of topological invariant, because:
$$ X \simeq Y \Rightarrow \Pg(X) \cong_{WL} \Pg(Y) $$
where the isomorphism of primigraph spaces is weak and local, in a sense that will be made precise later. This local isomorphism gives us a powerful tool, because $\Pg(X)$ is often a large space or it is difficult to compute explicitely, and thus it is often better to deal with their affine primigraph subspaces. We will indeed show that, if we consider any affine primigraph subspace of $\Pg(X)$, and there is no affine primigraph subspace of $\Pg(Y)$ which is isomorphic to it, then $X$ and $Y$ cannot be homeomorphic. This avoids the computation of $\Pg$ (at least for one of the spaces); here we will explicitely evaluate $\Pg$ in two basic examples: the first one involves the trivial topology, while the second one is a finite or countably infinite generalisation of the first one. We will also show some examples of applications of the invariance; in particular, a local version of $\Pg$, namely $\Pg_n$, can be really useful in establishing when two spaces are not homeomorphic, see the examples at the end of Section 4.\\
We refer to [16] and [41] for some useful results on computations of primitive ideal spaces of graph $C^*$-algebras, while we refer to [38-39] for some explicit examples of graph $C^*$-algebras. In the next Section we start to introduce all the needed concepts; moreover, we recall all the tools that will be used in the proofs of our main results. In Section 3 we prove our first main Theorem and we give the two basic examples outlined above. In Section 4 we give the proof of the second main Theorem; we then show some fruitful ways to apply this result, which avoid the computation of at least one of the $\Pg$'s. We conclude with some other invariants which arise from the proofs of the previous theorems.
%%%%%%%%%%%%%%%%%%%%%%%%%%%%%%%%%%%
\section{Preliminaries}
We start recalling the definition of the map $h$, as in [1-3]. This function describes "how dense" a point is in a topological space, w.r.t. some open covering. More precisely, let $X$ be a topological space, and let $(X_p)_p$ be some open covering of $X$, that is, a collection $(X_p)_p$ of open subsets of $X$ such that
$$\bigcup_p X_p = X$$
Following the definition in Section 4 of [1], but applied to topological spaces in a similar way to what we did in [2-3], we define $h:X \rightarrow \mathcal{L}$:
\begin{equation}\label{Eq:2.1}
h(x):=\lbrace X_t \in (X_p)_p : x \in X_t \rbrace
\end{equation}
Here $\mathcal{L}$ represents, as in [1-3], the power collection \footnote{The power collection is defined as for sets.} of $(X_p)_p$ without the empty collection. We define, as in [1-3], the following equivalence relation:
\begin{equation}\label{Eq:2.2}
x \sim y \Leftrightarrow h(x)=h(y)
\end{equation}
which gives rise to a partial order on $X/ \sim$. As usual, we fix some representatives so that $X/ \sim \, \subseteq X$. We now want to assign this poset a graph. Since many results for graph $C^*$-algebras hold for at most countably infinite graphs \footnote{For uncountable graphs, some of these results do hold (see for instance [15]). However, here we prefer to focus on countable graphs.}, we will not consider uncountably infinite posets. This means that we will only deal with 'admissible covers':
\begin{definition}\label{Def:2.1}
An open covering $(X_p)_p$ (which does not contain $\emptyset$) of a topological space $X$ is 'admissible' if the corresponding poset $X/ \sim$ is at most countably infinite.
\end{definition}
Notice that:
\begin{lemma}\label{Lm:2.1}
Every topological space $X \neq \emptyset$ has at least one admissible cover.
\end{lemma}
\begin{proof}
Just take $(X_p)_p \equiv \lbrace X \rbrace$: $X$ is open by definition of topology, it clearly covers itself and the corresponding poset has only one point.
\end{proof}
This will imply (see the proof in Section 3) that \textit{every} topological space has a corresponding primigraph space. We remark that the unique exception in Lemma \ref{Lm:2.1} is the case $X=\emptyset$, for which we consider as cover, by definition, $(X_p)_p=\lbrace \emptyset \rbrace$. In all the other cases, $\emptyset$ is not considered because it could be always added to any cover not containing it, and so it would not be useful for our purposes (actually, it would enlarge $\Pg$, making more difficult its computation). We will return to $\emptyset$ later.\\
Now recall the definition of directed graph \footnote{If not specified otherwise, when we say 'graph' we always refer to a directed graph here.} (see, for instance, [11-12]):
\begin{definition}\label{Def:2.2}
A directed graph $G$ is an ordered pair $(V,E)$, where $V$ is the set of vertices, and $E$ is the set of edges, that is, of ordered pairs of vertices.
\end{definition}
Actually, some authors only allow distinct vertices in the definition of edge. We will see below that, even with this definition, we will not have any problem.\\
Now, it is known that we can always assign a directed graph to every poset $P$. This can be done, for instance, as follows (using Hasse diagrams):\\
1) the set $V$ of vertices is $P$ itself;\\
2) the set $E$ of edges consists of the ordered pairs $(x,y)$ such that $x < y$ (that is, $x \leq y$ and $x \neq y$) and there is no $z \in X/ \sim$ for which $x < z < y$ \footnote{Notice that the assumption of dealing only with at most countably infinite posets allows us to do this without difficulties (in particular, verifying that there is no $z$ such that $x < z < y$ is simple, thanks to this hypothesis).};\\
2') if there are no $x, y \in P$ such that $z < x$ or $y < z$, then we will consider no edge for $z$.\\
We give some examples of constructions of graphs from posets. If $P$ consists of the points $x,y,z,w,v,$, with $x < z < y$ and $v < w$ and no other order relations, we have the following directed graph:
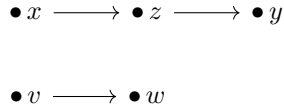
\begin{center}
\begin{tikzcd}
\bullet \, x \arrow{r} & \bullet \, z \arrow{r} & \bullet \, y \\
\bullet \, v \arrow{r} & \bullet \, w
\end{tikzcd}
\end{center}
Another kind of directed graph is the following one:
$$\bullet \, x$$
$$\bullet \, y$$
$$\bullet \, z$$
without "connections" between the vertices. Actually, we can also consider at least two other possibilities (we will however follow the above one in this paper). The first one modifies 2'), considering the edge $(z,z)$ instead of not considering any edge (this possibility is not allowed by some authors, as said above). Another kind of graph can be obtained as follows: if we consider the graph above, we can define edges:
\begin{center}
\begin{tikzcd}
\bullet \, x \arrow{d} \\
\bullet \, y \arrow{d} \\
\bullet \, z 
\end{tikzcd}
\end{center}
Something similar can be done in the general case.\\
Clearly, the above arguments can be applied to the poset $X/ \sim$ \footnote{Notice that the graphs obtained are different, in general, from 'intersection graphs'.} (which will be always assumed to be at most contable here, as already said above). It is also easy to see that the following result holds:
\begin{proposition}\label{Prop:2.1}
The graphs corresponding to different choices of representatives of $X / \sim$ are all isomorphic.
\end{proposition}
\begin{proof}
Recall that two graphs are isomorphic if there is a bijection between the sets of vertices, say $f$, and if such bijection is such that whenever $(x,y)$ is an edge in the first graph, $(f(x),f(y))$ is an edge in the second graph, and viceversa. It is then clear, by definition of $X/ \sim$, that all the graphs obtained are isomorphic.
\end{proof}
Thus, it does not matter how we choose the representatives, because we will see later that isomorphic graphs induce isomorphic graph $C^*$-algebras. We now recall some basic aspects of this theory. First of all, we recall the definition of $C^*$-algebra:
\begin{definition}\label{Def:2.3}
A $C^*$-algebra $A$ is a Banach algebra over $\mathbb{C}$, together with an involution (that is, a map $x \mapsto x^*$, with $x \in A$, such that $(x^*)^*=x^*$) satisfing, $\forall x, y \in A$ and $\forall z \in \mathbb{C}$:\\
$$ (x+y)^* = x^* + y^*$$
$$(xy)^*=y^* x^*$$
$$(z x)^* = \overline{z} x^*$$
$$\Vert x^* x \Vert = \Vert x \Vert \Vert x^* \Vert$$
\end{definition}
A $^*$-homomorphism of $C^*$-algebras is a bounded linear map $T:A \rightarrow B$ such that, $\forall a, b \in A$:
$$T(ab)=T(a)T(b)$$
$$T(a^*)=T(a)^*$$
A bijective $^*$-homomorphism of $C^*$-algebras is called $C^*$-isomorphism.\\
Sheaves of $C^*$-algebras are defined as usual (see [10]). Here we will always refer to the first category defined in [4], namely $C^* \textbf{alg}$. See also [8] and [17].\\
There is a way to assign a $C^*$-algebra to any directed graph $G$ (it is easy to see that our definition of directed graph is equivalent to the one commonly used in this context). This has been done first for row-finite directed graphs, and then extended to coutably infinite graphs. We refer to [18] for an introduction on these topics and for some history; see also [19-20]. For graph $C^*$-algebras of row finite graphs, see [22]; for infinite graphs, see [21] and [23]. The existence of such graph $C^*$-algebras for any (even countably infinite) graph has been established in various papers (see, for instance, the ones above). Some results also hold for uncountably infinite graphs (see [15]), as said before, but here we only concentrate on countable ones. Briefly, we deal with Cuntz-Krieger $G$-families, where $G$ is a graph; it can be shown that there exists a (unique, up to $^*$-isomorphism) universal $C^*$-algebra generated by such a family, and this is called graph $C^*$-algebra associated to $G$. This $C^*$-algebra is usually denoted by $C^*(G)$. An important result is the following one:
\begin{proposition}\label{Prop:2.2}
Let $G_1$, $G_2$ be two graphs. Then:
$$ G_1 \cong G_2 \Rightarrow C^*(G_1) \cong C^*(G_2) $$
\end{proposition}
\begin{proof}
See, for instance, the comments below Lemma 2.1.1 in [25]. The result also holds for countably infinite graphs.
\end{proof}
We now recall the notion of primitive spectrum (also called primitive ideal space), that is, the set consisting of primitive ideals of the $C^*$-algebra considered (a primitive ideal is the kernel of some irreducible representation). This space, denoted by $\Prim (\cdot)$, is endowed with the Jacobson topology (also called hull-kernel topology). See [16] for various results on the primitive spectrum of graph $C^*$-algebras. We refer to Section 3.1 of [26] for more details on this notion. The following result will be needed later:
\begin{proposition}\label{Prop:2.3}
If the $C^*$-algebras $A$ and $B$ are Morita equivalent, then $\Prim A \simeq \Prim B$ (their primitive spectra are homeomorphic).
\end{proposition}
\begin{proof}
This is a well known result. For instance, see the Introduction in [27] or [28] (before Theorem 2.3.3). Relatively to the latter one, we also refer to Corollary 3.3(a) in [29] for Rieffel homeomorphism between $\Prim A$ and $\Prim B$.
\end{proof}
Actually, we will use the following Corollary:
\begin{corollary}\label{Crl:2.1}
Let $A$, $B$ be $C^*$-algebras. Then, $A \cong B \Rightarrow \Prim A \simeq \Prim B$.
\end{corollary}
\begin{proof}
It is known (see, for instance, below Example 2.3 in [29]) that if two $C^*$-algebras are isomorphic, then they are also Morita equivalent. Thus, applying Proposition \ref{Prop:2.3} we conclude the proof.
\end{proof}
We can now proceed with the definition of primigraph space. We first need:
\begin{definition}\label{Def:2.4}
A $C^*$-algebraic space is a topological space $X$ together with a sheaf of $C^*$-algebras on it, where we consider the category $C^*$\textbf{alg} (see [4]).
\end{definition}
The notions of morphism and isomorphism between two $C^*$-algebraic spaces are analoguous to the ones for ringed spaces (see, for instance, [5-6]). The primitive spectrum can be turned into a $C^*$-algebraic space by using the following sheaf $\mathfrak{C}$ (see [8]):
$$ \mathfrak{C}(U):=C_b(U) $$
for all the open subsets $U$, where $C_b(Y)$ denotes the $C^*$-algebra in Example 2.1.8 of [9] (note that it is defined for \textit{every} topological space $Y$; we will use this fact in the proof of Theorem \ref{Thm:3.1}). Finally, we have:
\begin{definition}\label{Def:2.5}
An affine primigraph space is a $C^*$-algebraic space which is isomorphic, as a a $C^*$-algebraic space, to $\Prim A$, for some $C^*$-algebra $A$.\\
A primigraph space is a $C^*$-algebraic space which can be covered by affine primigraph spaces.
\end{definition}
It is clear that this new kind of space is analoguous to schemes, which are locally ringed spaces that can be covered by spaces which are isomorphic to the spectrum $\Spec R$ of some ring, and perfectoid spaces, which are covered by $\Spa(R,R^+)$'s; see [5-6] and [30-34] for more on these topics. Morphisms of primigraph spaces are morphisms of $C^*$-algebraic spaces.\\
The name of these new spaces comes from the primitive spectrum $\Prim A$ of a $C^*$-algebra, which is involved in the definition itself, and from the proof of Theorem \ref{Thm:3.1}, where we use graphs.\\
In the next section we prove our first main result, namely that it is possible to associate to every topological space a primigraph space. Actually, we will see that this is not unique, but we will consider one which is given by the union of all the affine primigraph spaces naturally obtained in our construction.
%%%%%%%%%%%%%%%%%%%%%%%%%%%%%%%%%%%%
\section{Primigraph space associated to a topological space}
We now associate to every topological space a certain primigraph space. Let $X$ be a topological space; choose any admissible cover for this space (we know that at least one admissible cover always exists by Lemma \ref{Lm:2.1}) \footnote{Actually, we have as unique exception the empty space. This will be treated in Section 4, where we show that a primigraph space can be also assigned to $\emptyset$.}. Then, by \eqref{Eq:2.1} and \eqref{Eq:2.2} we have a poset $X/ \sim$. We have shown in the previous section that we can assign it a directed graph (recall that $X/ \sim$ is at most countable by Definition \ref{Def:2.1}). We denote again by $X/ \sim$ this graph. Now, consider the graph $C^*$-algebra $C^*(X/ \sim)$, which exists and is unique up to $^*$-isomorphism. Notice that this $C^*$-algebra does not depend, up to isomorphism, on the chosen representatives of $X/ \sim$, because of Proposition \ref{Prop:2.1} together with Proposition \ref{Prop:2.2}. We can then consider the primitive spectrum of this $C^*$-algebra (we will show later, in the proofs of Theorem \ref{Thm:4.1} and Lemma \ref{Lm:4.1}, that isomorphic graphs give rise to isomorphic, as $C^*$-algebraic spaces, primitive spectra). Now consider the space:
\begin{equation}
\Pg(X):=\bigcup_p \Prim(A_p)
\end{equation}
where 
$$ A_p := X/ \sim_p $$
and the union is over all the admissible covers of $X$. We only need to give this space a topology and a sheaf of $C^*$-algebras, in such a way that it becomes a primigraph space. We define the topology as the smallest one generated by the collection (see, for instance, Example 1.1 in [1]):
$$ \lbrace U, \, \text{where} \,  U \,  \text{is an open subset of some} \, \Prim A_p \rbrace $$
We use again the sheaf $\mathfrak{C}$, but this time on $\Pg(X)$. Now it only remains to verify that this is compatible with the sheaves on each $\Prim A_p$. But, by definition of our topology, the open subsets of some $\Prim A_p$ are also open in $\Pg(X)$, and we clearly have (by the definition itself of $\mathfrak{C}$):
$$ \mathfrak{C}_{\Pg(X)}(U) = \mathfrak{C}_{\Prim A_p}(U)$$
where the subscript indicates which sheaf we are referring to, and where $U$ is an open subset of $\Prim A_p$ (and thus also of $\Pg(X)$). Therefore, we have completed the proof of our first main result:
\begin{theorem}\label{Thm:3.1}
To every topological space $X$ there corresponds a primigraph space $\Pg(X)$.
\end{theorem}
We now give some simple examples of computation of $\Pg$. For the evaluation of some graph $C^*$-algebras, we refer to [38-39]. Since the explicit evaluation of $\Pg$ is not easy in general \footnote{This is why the invariants involving subspaces are often useful: we only have to find a subspace of $\Pg(X)$ which is not isomorphic to any subspace of $\Pg(Y)$ (or viceversa) to conclude that $X$ is not homeomorphic to $Y$. See the next Section for more details and for some examples.}, here we will only consider two basic cases: the first one is the trivial one, while the second one is a generalisation of the first one.
\begin{example}\label{Ex:3.1}
\normalfont Consider any topological space $X$ endowed with the trivial topology $\tau:=\lbrace X, \emptyset \rbrace$. Then, the unique corresponding (admissible) poset is a singleton, and hence we obtain the following graph:
$$ \bullet $$
It is well known that its graph $C^*$-algebra is $\mathbb{C}$, and thus:
$$ \Pg(X)= \Prim \mathbb{C} $$
(together with Jacobson topology and the usual sheaf of bounded functions).
\end{example}
\begin{example}\label{Ex:3.2}
\normalfont We generalise the previous example with the following topology. Consider any space $X$, and define:
$$ \tau:=\lbrace \emptyset, U_1, U_2, ..., U_n \rbrace $$
where
$$ U_1 \subsetneq U_2 \subsetneq ... \subsetneq U_n=X $$
for some sets $U_k$ ($n \in \mathbb{N} \setminus \lbrace 0 \rbrace$). If $n=1$, we obtain the trivial topology, as in the previous example. It is clear that, up to isomorphism, we have precisely these graphs (obtained from all the possible admissible covers \footnote{For example, the first graph is obtained from the cover $\lbrace X \rbrace$, while the second one is obtained from any cover of the kind $\lbrace X, U_k \rbrace$, with $k \neq n$ (note that different $k$'s give rise to isomorphic graphs).}):
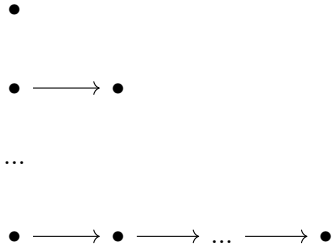
\begin{center}
\begin{tikzcd}
\bullet \\
\bullet \arrow{r} & \bullet \\
... \\
\bullet \arrow{r} & \bullet \arrow{r} & ... \arrow{r} & \bullet
\end{tikzcd}
\end{center}
where in the last line we have $n$ points. The graph $C^*$-algebra corresponding to the $k$-th graph above is $M_k(\mathbb{C})$, that is, the $C^*$-algebra of $k \times k$ matrices with values in $\mathbb{C}$. Then, we have:
$$ \Pg(X)= \bigcup_{k=1}^{n} \Prim M_k(\mathbb{C}) $$
When we have a countably infinite number of $U_k$'s, we can easily extend this result to:
$$ \Pg(X)= \bigcup_{k \in \mathbb{N} \setminus \lbrace 0 \rbrace} \Prim M_k(\mathbb{C}) \, \cup \, \Prim B_0(H)$$
where $B_0(H)$ is the $C^*$-algebra of compact operators on a separable infinite dimensional Hilbert space $H$ (recall that every separable infinite dimensional Hilbert space is isometrically isomorphic to $l^2$). It is known that $B_0(H)$ can be seen as a limit of $M_n(\mathbb{C})$, where we consider $M_n(\mathbb{C}) \subset M_{n+1}(\mathbb{C})$ in the usual way. Thus, we can say that:
$$ \Pg(X)=\Prim B_0(H) $$
\end{example}
%%%%%%%%%%%%%%%%%%%%%%%%%%%%%%%%%%%%%%%%%
\section{Topological invariants}
In this section we prove our second main result, namely:
\begin{theorem}\label{Thm:4.1}
If $X \simeq Y$ (i.e. $X$ and $Y$ are homeomorphic), then $\Pg(X) \cong_{WL} Pg(Y)$.
\end{theorem}
The 'weak local' isomorphism will be useful in the applications of this result, as we will show later. Note that the above Theorem, together with Theorem \ref{Thm:3.1}, gives us a new topological invariant. This is however difficult to compute in general, as noticed in the previous section. The local application of the result can thus be really helpful in various situations.
\begin{proof}
Suppose that $X$ is homeomorphic to $Y$ ($X \simeq Y$). Then, chosen any admissible covering of $X$, we can always find an admissible open covering of $Y$ as follows (the same can be done starting, conversely, from an admissible cover of $Y$ and using $\pi^{-1}$). If $\pi : X \rightarrow Y$ is a homeomorphism, then
$$ (\pi(X_p))_p $$
is an admissible cover of $Y$. Indeed, since $\pi$ is continuous we know that $\pi(X_p)$ is open for all $p$. Moreover, since $\pi$ is also bijective it is clear that:
$$ \pi(\bigcup_p X_p) = \bigcup_p \pi(X_p) $$
and thus:
$$ Y=\pi(X)=\pi(\bigcup_p X_p)= \bigcup_p \pi(X_p)$$
from which we conclude that $(\pi(X_p))_p$ is an open cover. It is admissible because:
$$h_X(x)=\lbrace \text{some} \, X_t \text{'s} \rbrace \Leftrightarrow h_Y(\pi(x))=\lbrace \text{the corresponding} \, \pi(X_t) \text{'s} \rbrace$$
and because, by assumption, the chosen cover of $X$ is admissible. Now, for each open covering $(X_p)_p$ of $X$, consider some homeomorphism $\pi$, and notice that the graphs obtained from $X/ \sim$ (where we used $(X_p)_p$) and $Y/ \sim$ \footnote{We are abusing notation and writing $\sim$ for both. We can do this thanks to homeomorphism.} (where we used $(\pi(X_p))_p$) are isomorphic. This easily follows from the above equation connecting $h_X$ and $h_Y$ an from the fact that:
$$ \pi(X_t \cap X_p)= \pi(X_t) \cap \pi(X_p)$$
(this holds because $\pi$ is bijective). Now suppose that $g$ is another homeomorphism between $X$ and $Y$. We can do the same, and the graph $Y/ \sim$ obtained this time turns out to be again isomorphic to the previous one (thus justifing the same notation for both). Indeed, again because of the previous relation, the posets are isomorphic (as posets), because there is clearly a bijection between the elements (thanks to the fact that homeomorphisms are bijections) and the partial orders are maintained. Thus, as we did in the previous proof, we can conclude that we always obtain isomorphic graphs. Consequently, chosen any admissible cover of $X$, we have 
$$X/ \sim \, \, \cong Y/ \sim$$
(as graphs). Hence, we obtain two isomorphic graph $C^*$-algebras, which give rise to two homeomorphic primitive spectra 
$$\Prim \, _X A_j \simeq \Prim \, _Y B_j$$
(we have reordered the indeces). This follows from Proposition \ref{Prop:2.2} and Corollary \ref{Crl:2.1}. If 
$$\Prim \, _X A_j \cong \Prim \, _Y B_j$$
also as $C^*$-algebras spaces, then it is clear that $\Pg(X)$ and $\Pg(Y)$ arise from unions of isomorphic $C^*$-algebraic spaces. We hence only have to prove that:
$$\Prim \, _X A_j \simeq \Prim \, _Y B_j \Rightarrow \Prim \, _X A_j \cong \Prim \, _Y B_j$$
and this is done below.
\end{proof}
Before concluding the proof, we reformulate this last part as follows:
\begin{lemma}\label{Lm:4.1}
If $X \simeq Y$, then $\Prim C^*(X/ \sim) \cong \Prim C^*(Y/ \sim)$, where the (admissible) covering of $X$ is given by some $(X_p)_p$, and the covering of $Y$ is given by $(\pi(X_p))_p$, for \textbf{any} homeomorphism $\pi: X \rightarrow Y$.
\end{lemma}
We notice that, for computational purposes, the topological invariant of the affine primigraph space in this Lemma is simpler than $\Pg$ (which is usually large). We refer, again, to [16] and [41] for some results on the evaluation of primitive spectra of graph $C^*$-algebras.
\begin{proof}
By the previous argument, we only have to prove that
$$\Prim C^*(X/ \sim) \simeq \Prim C^*(Y/ \sim) \Rightarrow \Prim C^*(X/ \sim) \cong \Prim C^*(Y/ \sim)$$
Notice that, by what we proved above, the result holds for \textbf{any} homeomorphism that gives rise to some covering of $Y$. The homeomorphism (say, $\pi$) is the map between the underlying topological spaces of the $C^*$-algebraic spaces involved. Now, thanks to homeomorphism, we know that $\pi^{-1}(U)$ is open in $\Prim C^*(X/ \sim)$ whenever $U$ is open in $\Prim C^*(Y/ \sim)$. Moreover, we also know that $\pi$ is a bijection. We find a linear map between $\mathfrak{C}(U)$ and $\mathfrak{C}(\pi^{-1}(U))$ as follows. Since these $C^*$-algebras constist of bounded functions on $U$, $\pi^{-1}(U)$, respectively, to each map $s \in \mathfrak{C}(U)$ we can assign the map $\widetilde{s} \in \mathfrak{C}(\pi^{-1}(U))$ defined as follows:
$$ \widetilde{s}(t):= s(\pi^{-1} (t))$$
This is possible thanks to the bijectivity of $\pi$. It is clear that there is a bijection between $\mathfrak{C}(U)$ and $\mathfrak{C}(\pi^{-1}(U))$ obtained with the mapping above (which we denote by $T$). Indeed,
$$ s \xmapsto{T} \widetilde{s} $$
is certainly surjective. To prove injectivity, suppose that $\widetilde{s}_1 = \widetilde{s}_2$; we have to show that $s_1 = s_2$. Now, we have by definition:
$$ s_1 \circ \pi^{-1} = s_2 \circ \pi^{-1}$$
$\pi$ is bijective and thus invertible. Therefore, we can do the following:
$$ s_1 \circ \pi^{-1} \circ \pi = s_2 \circ \pi^{-1} \circ \pi $$
and consequently we have $s_1=s_2$, concluding the proof of the bijectivity of $T$. Moreover, this mapping is clearly linear. It is also bounded by definition of $\mathfrak{C}$, which consists of bounded functions. To prove that it is a $C^*$-isomorphism, we only have to check that:
$$T(ab)=T(a)T(b)$$
$$T(a^*)=T(a)^*$$
The first one is easily verified; for the second one, recall that by definition of $\mathfrak{C}$ the involution considered is conjugation, and so even this equation follows. To conclude the proof, we need to verify that the usual square commutes. But this can be easily shown, because the operation of restriction and the application of $T$ can be interchanged.
\end{proof}
We now conclude the proof of Theorem \ref{Thm:4.1}. We define:
\begin{definition}\label{Def:4.1}
Let $\Pg(X),\Pg(Y) \in \Pg(\Ob(\textbf{Top}))$, meaning that these primigraph spaces are associated to some topological space ($X$, $Y$, respectively) as in the proof of Theorem \ref{Thm:3.1}. $\Pg(X)$ and $\Pg(Y)$ are said to be weakly locally isomorphic, in symbols:
$$\Pg(X) \cong_{WL} \Pg(Y)$$
if the following two conditions holds:\\
1) for every admissible cover of $X$ (denoted by '$t$') there exists at least one admissible cover of $Y$ (denoted by '$p$') such that:
$$ \Prim(X/ \sim_t) \cong \Prim(Y/ \sim_p) $$
(as $C^*$-algebraic spaces);\\
2) conversely, for every admissible cover of $Y$ (denoted by '$t$') there exists at least one admissible cover of $X$ (denoted by '$p$') such that:
$$\Prim(Y/ \sim_t) \cong \Prim(X/ \sim_p) $$
(as $C^*$-algebraic spaces).
\end{definition}
Intuitively, this means that the two primigraph spaces obtained from some topological spaces are covered by isomorphic affine primigraph spaces (where these affine primigraph spaces are the ones constructed in the proof of Theorem \ref{Thm:3.1}, namely $\Prim( \cdot / \sim_p)$ for some admissible cover).
\begin{remark}\label{Rm:4.1}
\normalfont Notice that '$\cong_{WL}$' is an equivalence relation. Indeed, $\Pg(X) \cong_{WL} \Pg(X)$ because for every admissible cover of $X$, we can choose the same cover and obtain:
$$ \Prim(X/ \sim_t) \cong \Prim(X/ \sim_t) $$
Moreover, if $\Pg(X) \cong_{WL} \Pg(Y)$, then $\Pg(Y) \cong_{WL} \Pg(X)$ because both 1) and 2) above hold. To conclude, it is easy to see that $\cong_{WL}$ is also transitive, thanks to the fact that the local isomorphisms '$\cong$' are equivalence relations.
\end{remark}
Clearly, if $X \simeq Y$, then every $\Prim (X/ \sim_t) \subseteq \Pg(X)$  is isomorphic, as a $C^*$-algebraic space, to some $\Prim (Y/ \sim_p) \subseteq \Pg(Y)$, and also the converse holds (we proved this above). Thus:
$$ \Pg(X) \cong_{WL} \Pg(Y)$$
and the Theorem follows.
\begin{remark}\label{Rm:4.2}
\normalfont $\Pg$ is a new kind of topological invariant: if we can find some subspace $\Prim A$ in $\Pg(X)$ which is not isomorphic to any $\Prim B$ in $\Pg(Y)$, we can immediately conclude that $X$ and $Y$ are not homeomorphic. Furthermore, we can also use Lemma \ref{Lm:4.1}, which shows us that $\Prim$ gives rise (under all the mappings considered) to a topological invariant. This is easier to compute, as already noticed in this paper. We think that one of the most interesting aspect of primigraph spaces is that every topological space has a natural space of this kind associated to it, and this also constitutes a topological invariant. In fact, our main motivation for introducing these spaces was precisely this one. We also think that the results in [16] and [41] could be fruitfully applied to the previous theorems, in order to be able to verify, in practice, when two spaces are not homeomorphic. It could be difficult to evaluate explicitely $\Pg$, but various results together with the ones in this Section could at least lead to something easier to check.
\end{remark}
Now we give a really simple example of application using Example \ref{Ex:3.1} and Example \ref{Ex:3.2}. Of course, due to the fact that these examples involve the trivial topology and a slight generalisation of it, we obtain an obvious and well known conclusion. However, this example can be useful to see how to apply the previous results in more general and interesting cases.
\begin{example}\label{Ex:4.1}
\normalfont A topological space endowed with the trivial topology is not homeomorphic to any topological space with a topology as the one in Example \ref{Ex:3.2} with $n \neq 1$. Indeed, indicating the two spaces by $X$ and $Y$, respectively, we have:
$$ \Pg(X) \not \cong_{WL} \Pg(Y)$$
by the computations in the previous examples. Applying Theorem \ref{Thm:4.1}, we conclude the proof of the above statement. We could also use a "local version" of the theorem to arrive at this conclusion as follows. $\Prim \mathbb{C}$ is isomorphic to the affine primigraph subspace $\Prim \mathbb{C} \subset \Pg(Y)$, so we cannot conclude anything. However, looking in the "other direction", it is clear that $\Prim M_n(\mathbb{C})$ is not isomorphic to $\Prim \mathbb{C}$ for $n \neq 1$ (a similar argument holds for the countable case), and thus we can conclude that the spaces are not homeomorphic.
\end{example}
As the previous example shows, Theorem \ref{Thm:4.1} can be used in a "local form", which can be essentially obtained via "decomposition". We do not need to compute both $\Pg$'s in general: if we are able to find some information about one of them (not necessarily computing it explicitely), and we can find an affine primigraph subspace of the other one which is not isomorphic to any affine primigraph subspace $\Prim ( \cdot / \sim_p)$ of the former, then we can immediately conclude that the two spaces considered are not homeomorphic. Actually, a similar idea gives rise to another useful invariant . We define $\Pg_n(X)$ as before, but this time using only admissible covers consisting of precisely $n$ elements (if $n=\infty$, we consider it countable). By Lemma \ref{Lm:4.1} \footnote{Note that the fact that homeomorphisms give rise to admissible coverings with the same number of elements of the former cover is fundamental in establishing invariance.}, we clearly have:
\begin{equation}\label{Eq:4.1}
X \simeq Y \Rightarrow \Pg_n(X) \cong_{WL} \Pg_n(Y), \quad \forall n \in (\mathbb{N} \cup \lbrace \infty \rbrace) \setminus \lbrace 0 \rbrace
\end{equation}
Thus, we can restrict our attention to covers with a chosen number of elements; this way, we have an invariant which is easier to compute, since it is smaller than $\Pg$. In fact, it is clear that:
\begin{equation}\label{Eq:4.2}
\Pg(X)=\bigcup_{n \in (\mathbb{N} \cup \lbrace \infty \rbrace) \setminus \lbrace 0 \rbrace} \Pg_n(X)
\end{equation}
Example \ref{Ex:4.1} can then be easily reobtained using $\Pg_n$ instead of $\Pg$. We notice that $\Pg_1$ can be always computed, because the unique admissible cover with only one element is $\lbrace X \rbrace$, for which we clearly have the graph
$$\bullet$$
whose graph $C^*$-algebra is $\mathbb{C}$, as already said. Thus, we can conclude that, whichever are the topological spaces $X, Y \in \Ob(\textbf{Top})$, we have:
\begin{equation}\label{Eq:4.3}
\Pg_1(X) \cong \Pg_1(Y) \cong \Prim \mathbb{C}
\end{equation}
(notice that here we do not only have a weak local isomorphism). To $X= \emptyset$, which had been excluded in Lemma \ref{Lm:2.1}, we then assign by definition the graph
$$\bullet$$
and hence we have
$$ \Pg(\emptyset)=\Prim \mathbb{C} $$
(because its unique cover is $\lbrace \emptyset \rbrace$, even though not admissible). This way, we have a primigraph space associated to every topological space $X$, concluding the unique remaining case in the proof of Theorem \ref{Thm:3.1}. We now give some examples (which also involve $\Pg_n$):
\begin{example}\label{Ex:4.2}
\normalfont We show that $I=[0, 2 \pi)$ (with the subspace topology induced by $\mathbb{R}$) is not homeomorphic to the unit circle 
$$S^1= \lbrace (x,y) \in \mathbb{R}^2 : x^2 + y^2 =1 \rbrace$$
(with the subspace topology induced by $\mathbb{R}^2$). Here, the Euclidean spaces are endowed with the usual topologies. We will prove this by showing that there is no open cover $(X_p)_p$ of $I$ consisting of four sets which is isomorphic to a particular cover (with four sets) of the unit circle. The reason will be due to the fact that the circle induces a circular behavior in the graph. Any open cover by four sets of $I$ consists of elements of one of the forms:
$$(a,b)$$
$$[0,c)$$
(the second interval is not open \textit{in} $\mathbb{R}$, but it is open \textit{in} $I$ because it is given by the intersection of $I$ with some open set in $\mathbb{R}$ (and thus it is open by definition of subspace topology)). We now show which are the graphs obtained in this situation. Let
$$ 0 < b < a < d < c < f < e < 2 \pi $$
and consider the admissible cover
$$ \lbrace [0,a), (b,c), (d,e), (f, 2 \pi) \rbrace $$
This gives rise to the following graph:
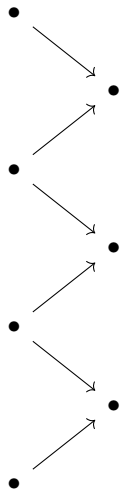
\begin{center}
\begin{tikzcd}
\bullet \arrow{dr} \\
& \bullet \\
\bullet \arrow{ru} \arrow{dr} \\
& \bullet \\
\bullet \arrow{ru} \arrow{dr} \\
& \bullet \\
\bullet \arrow{ru}
\end{tikzcd}
\end{center}
The following modified inequality
$$ 0 < b < d < a < f < c < e < 2 \pi $$
with the same cover as above leads to the graph:
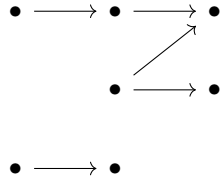
\begin{center}
\begin{tikzcd}
\bullet \arrow{r} & \bullet \arrow{r} & \bullet\\
& \bullet \arrow{ru} \arrow{r} & \bullet \\
\bullet \arrow{r} & \bullet
\end{tikzcd}
\end{center}
If instead we considered a covering containing $I$, say, for instance:
$$ \lbrace [0,2 \pi), (a,b), (c,d), (e,f) \rbrace $$
with
$$ 0 < a < c < b < e < d < f < 2 \pi $$
then we would obtain:
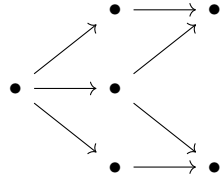
\begin{center}
\begin{tikzcd}
& \bullet \arrow{r} & \bullet\\
\bullet \arrow{ru} \arrow{r} \arrow{dr} & \bullet \arrow{ru} \arrow{dr}\\
& \bullet \arrow{r} & \bullet
\end{tikzcd}
\end{center}
Now consider a particular cover of the circle, namely:
$$ X_p:=\lbrace (x,y) \in S^1 : \Arg (x+iy) \in (p \frac{\pi}{2} - \frac{1}{100}, (p+1) \frac{\pi}{2} + \frac{1}{100}) \rbrace $$
where $\Arg$ denotes the principal argument \footnote{Actually, we have abused notation using some values $ \geq 2 \pi$ for the principal argument, which is not allowed. However, the cover of the circle we are dealing with should be clear.} of a complex number, and $p=0, 1,2,3$. The graph obtained from this cover is:
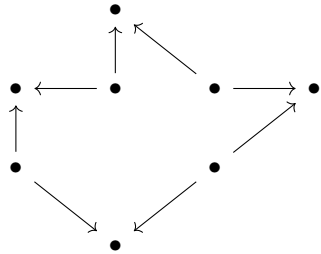
\begin{center}
\begin{tikzcd}
& \bullet \\
\bullet & \bullet \arrow{l} \arrow{u} & \bullet \arrow{ul} \arrow{r} & \bullet \\
\bullet \arrow{u} \arrow{dr} & & \bullet \arrow{dl} \arrow{ru} \\
& \bullet
\end{tikzcd}
\end{center}
Clearly, no cover with four elements of $I$ can have the "circular" behavior in the previous graph. The first cover of $I$ shown above is almost the same, because there is only a "missing" arrow: however, suppose that such an arrow could be found modifying the cover of $I$. This would imply that 
$$[0,a) \cap (f, 2 \pi) \neq \emptyset $$
but if this were true, we would loose the two "middle points" of the graphs (because 
$$ \lbrace [0,a), (f, 2 \pi) \rbrace$$
would be itself a cover of $I$), obtaining, for instance, something like:
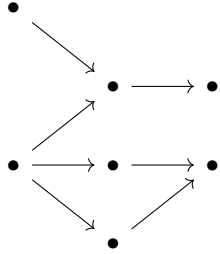
\begin{center}
\begin{tikzcd}
\bullet \arrow{dr} \\
& \bullet \arrow{r} & \bullet \\
\bullet \arrow{ru} \arrow{dr} \arrow{r} & \bullet \arrow{r} & \bullet\\
& \bullet \arrow{ru}
\end{tikzcd}
\end{center}
(here we have used the inequality:
$$ 0 < f < b < a < d < c < e < 2 \pi $$
which is a slight modification of the above one, where we have $f < a$ so that the intersection is not empty) and thus we would not arrive at the same graph. It is easy to see that no graph for a cover with $n=4$ elements of $I$ is isomorphic to the chosen graph of $S^1$, and thus there cannot be a homeomorphism, because if $\pi$ were a homeomorphism the cover $(\pi(X_p))_p$ of $I$ would have $4$ elements and would give rise to a graph isomorphic to the former one. Since this does not happen, we conclude that
$$ S^1 \not \simeq [0,2 \pi)=I$$
Notice that we have not even used primigraph spaces here: we have only used the first step in the proof of the invariance of these spaces (Theorem \ref{Thm:4.1}). By using some arguments involving the properties of graph $C^*$-algebras induced by their corresponding graphs (see, for instance, Section 2 in [42]), we can also arrive at the same conclusion using $\Pg_4$.
\end{example}
\begin{example}\label{Ex:4.3}
\normalfont We consider $\mathbb{R}$ and $\mathbb{R}^2$ with the usual topologies. Again, we let $n=4$, so that the covers of $\mathbb{R}$ can be obtained similarly to the ones in the previous example. Since the graphs obtained are similar, we do not represent them here. We consider a particular cover of $\mathbb{R}^2$ and we show, as before, that this cannot give rise to a graph isomorphic to any of the above ones. Again, arguments on the properties of graph $C^*$-algebras (see [42] for some examples) allow us to conclude also using $\Pg_4$. The cover we consider is the following one:
$$ X_1 := \lbrace (x,y) \in \mathbb{R}^2 : x < 8, y > -6 \rbrace $$
$$ X_2 := \lbrace (x,y) \in \mathbb{R}^2 : x > -3, y < 4 \rbrace $$
$$ X_3:= \lbrace (x,y) \in \mathbb{R}^2 : x,y > 0 \rbrace $$
$$ X_4 := \lbrace (x,y) \in \mathbb{R}^2 : x < 4, y < 2 \rbrace $$
Then, after some calculations, we obtain the following graph:
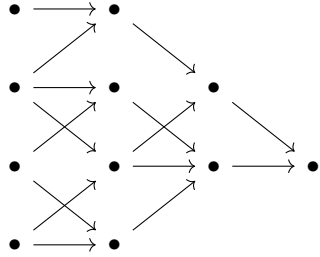
\begin{center}
\begin{tikzcd}
\bullet \arrow{r} & \bullet \arrow{dr} \\
\bullet \arrow{ru} \arrow{r} \arrow{dr} & \bullet \arrow{dr} & \bullet \arrow{dr} \\
\bullet \arrow{ru} \arrow{dr} & \bullet \arrow{ru} \arrow{r} & \bullet \arrow{r} & \bullet \\
\bullet \arrow{ru} \arrow{r} & \bullet \arrow{ru}
\end{tikzcd}
\end{center}
Clearly, no graph of $\mathbb{R}$ is isomorphic to the graph above, so we can immeditaly conclude that:
$$ \mathbb{R} \not \simeq \mathbb{R}^2 $$
Arguments as the ones in [42] lead to the same conclusion, using $\Pg_4$.
\end{example}
\begin{example}\label{Ex:4.4}
\normalfont We now consider $X=(\mathbb{R}, \tau_e)$, that is, $\mathbb{R}$ with the usual topology, and $Y=(\mathbb{R}, \tau)$, where $\tau$ is the topology generated by $(-\infty,0)$, $(0,\infty)$, $\lbrace 0 \rbrace$, that is:
$$ \tau:= \lbrace \mathbb{R}, (-\infty,0), (0,+\infty), \lbrace 0 \rbrace, [0, +\infty), (-\infty,0], \mathbb{R} \setminus \lbrace 0 \rbrace \rbrace $$
We can immediately conclude that these spaces are not homeomorphic because $\Pg_n(Y)$, for $n \geq 8$, is the empty set (because there is no cover with more than $7$ elements for $Y$) while $\Pg_n(X)$ is clearly nonempty. We also prove this in some other ways. Consider $\Pg_7$ for both spaces. For $Y$, we have only one possible cover with $7$ elements, namely $\tau$ itself. It is not difficult to see that:
$$h_Y(x)= \begin{cases} \lbrace \mathbb{R}, \mathbb{R} \setminus \lbrace 0 \rbrace, (0,+\infty), [0, +\infty) \rbrace, & \text{if} \, x > 0 \  \\ \lbrace \mathbb{R}, \lbrace 0 \rbrace, (-\infty,0], [0, +\infty) \rbrace, & \text{if} \, x=0 \ \\ \lbrace \mathbb{R}, \mathbb{R} \setminus \lbrace 0 \rbrace, (-\infty,0], (-\infty,0) \rbrace, & \text{if} \, x<0 \ \  \end{cases} $$
Thus, the graph obtained is:
$$ \bullet $$
$$ \bullet $$
$$ \bullet $$
Now consider the following cover with $7$ elements of $X$:
$$ \lbrace \mathbb{R}, (a,b), (c,d), (a_1,b_1), (c_1,d_1), (a_2,b_2), (c_2,d_2) \rbrace $$
where:
$$ - \infty < a < a_1 < a_2 < b_1 < c < c_2 < b_2 < b < c_1 < d_2 < d_1 < d < +\infty $$
After some calculations, we find the following graph:
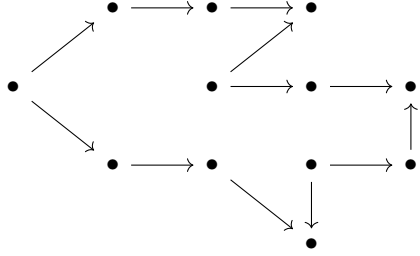
\begin{center}
\begin{tikzcd}
& \bullet \arrow{r} & \bullet \arrow{r} & \bullet \\
\bullet \arrow{ru} \arrow{dr} & & \bullet \arrow{ru} \arrow{r} & \bullet \arrow{r} & \bullet \\
& \bullet \arrow{r} & \bullet \arrow{dr} & \bullet \arrow{d} \arrow{r} & \bullet \arrow{u} \\
& & & \bullet
\end{tikzcd}
\end{center}
Thus, we can conclude that
$$ X \not \simeq Y$$
in at least three ways now (in addition to the previous one with $\Pg_n$, $n \geq 8$):\\
1) since there is only one cover with $7$ elements for $Y$, while $X$ has many other covers with $7$ elements, if there were a homeomorphism we would have that all the graphs obtained from the covers of $X$ would be all isomorphic, but this is clearly false, so the spaces cannot be homeomorphic;\\
2) we have found a graph (obtained from $X$, with $n=7$) which is not isomorphic to any (actually, the unique, up to isomorphism) graph obtained from $Y$ (with $n=7$). Thus, the spaces cannot be homeomorphic, as shown in the previous examples;\\
3) it is not difficult to see that the graph $C^*$-algebras of the graph obtained from $Y$ and the one for $X$ above are not isomorphic, and that neither their primitive spectra are isomorphic as $C^*$-algebraic spaces. Hence:
$$ \Pg_n(X) \not \cong_{WL} \Pg_n(Y) $$
and thus the spaces cannot be homeomorphic.
\end{example}
%%%%%%%%%%%%%%%%%%%%%%%%%%%%%%%%%%%%%%%
\subsection{Other invariants arising from some of the previous proofs}
Some of the proofs in the previous Section show that there are also other possible invariants for topological spaces that arise from the mapping:
$$ \text{topological space} \, \mapsto \, \text{graph(s)} \, \mapsto \, C^* \text{-algebra(s)} $$
These invariants involve the theory of structured spaces (see [1]). The first one is defined as follows. Consider the union of all the graph $C^*$-algebras over the admissible covers of $X$:
\begin{equation}\label{Eq:4.4}
\Str_{C^*}(X):= \bigcup_p C^*(X/ \sim_p)
\end{equation}
It is clear that this is a structured space, where all the fixed neighborhoods are $C^*$-algebras and the topology is, as usual, the one in Example 1.1 of [1], that is, the smallest one generated by this collection of sets. Now suppose that $X \simeq Y$. By the proof in the previous section, we know that this implies $C^*(X/ \sim_t) \cong C^*(Y/ \sim_t)$ for every admissible cover. Thus, by Section 3 in [1], we have:
$$ \Str_{C^*}(X) \cong \Str_{C^*}(Y) $$
(as structured spaces). Consequently, $\Str_{C^*}$ is a topological invariant:
\begin{proposition}\label{Prop:4.1}
$$ X \simeq Y \Rightarrow \Str_{C^*}(X) \cong \Str_{C^*}(Y) $$
where the second is an isomorphism of structured spaces.
\end{proposition}
Now we recall a well known invariant of $C^*$-algebras (or, more generally, of operator algebras). We refer to [34-37] for more details. Operator K-theory is an important invariant for $C^*$-algebras; it is known that, by Bott Periodicity Theorem, there are "only" two K-groups up to isomorphism, namely $K_0$ and $K_1$. A useful way to compute these groups is via a $6$-terms exact cyclic sequence, see for instance Section 6.3.1 in [37] or Theorem 12.1.2 in [34]. Now, to each graph $C^*$-algebra $A$ obtained from $X$ we assign the couple:
$$ (K_0(A),K_1(A)) $$
which is clearly an invariant. Since $K_0$ and $K_1$ are groups, we can define the structured space (see Section 3 in [1] for products of structured spaces):
\begin{equation}\label{Eq:4.5}
\Str_{K}(X):= \bigcup_p (K_0(X/ \sim_p),K_1(X/ \sim_p))
\end{equation}
(the union is again over all the admissible covers of $X$). Then, as for the previous invariant, it is easy to see that the following result holds:
\begin{proposition}\label{Prop:4.2}
$$ X \simeq Y \Rightarrow \Str_{K}(X) \cong \Str_{K}(Y) $$
where the second is an isomorphism of structured spaces.
\end{proposition}
Thus, we have obtained new interesting invariants for topological spaces. Even in these cases, it is often useful (from a computational point of view) to consider subspaces: for instance, if we are able to find some graph $C^*$-algebra given by $X$ which is not isomorphic to any $C^*$-algebra given by $Y$, then $X$ and $Y$ cannot be homeomorphic. This way, we avoid the evaluation of the (often large) space $\Str_{C^*}$. The restriction to some $n$ as for $\Pg_n$ is still possible, as it can be easily seen. These can be used, for instance, to give yet another proof of the results in the previous examples.
%%%%%%%%%%%%%%%%%%%%%%%%%%%%%%%%%
\section{Conclusion}
In this paper we have introduced a new kind of space (similar, in some sense, to schemes and perfectoid spaces) and we have shown that every topological space gives rise in a natural way to one of these spaces. A remarkable property of these spaces is that they are also a topological invariant: if $\Pg(X)$ is not weakly locally isomorphic to $\Pg(Y)$, then $X$ and $Y$ cannot be homeomorphic (simpler invariants are obtained via a "local form" of Theorem \ref{Thm:4.1} or by using $\Pg_n$). Following another direction, we note that primigraph spaces are also related to Leavitt path algebras (see, for instance, [40]): we think that these new spaces could be also useful to study the connections between such algebras and graph $C^*$-algebras. We also think that, even though $\Pg$ is often difficult to compute, using results such as the ones in [38-39], [16] and [41], it would be possible to find some expressions for them, or at least we may be able to find something that, together with the results in Section 4, could give easier methods to apply the new invariants introduced here.

\end{document}